\documentclass{amsart}

\usepackage{amssymb}
\usepackage{xcolor}
\usepackage{amsmath}
\usepackage{bm}
\usepackage{tkz-euclide}
\usepackage{graphicx}
\usepackage{caption}

\newcommand{\half}{\frac{1}{2}}
\newcommand{\thalf}{\tfrac{1}{2}}

\newcommand{\sume}{\mathop{{\sum}^{+}}}
\newcommand{\sumo}{\mathop{{\sum}^{-}}}

\newcommand{\intt}{\int_{-\infty}^{\infty}}

\newcommand{\mybinom}[2]{\Bigl(\begin{array}{@{}c@{}}#1\\#2\end{array}\Bigr)}

\newcommand{\thalfps}{\thalf + s}
\newcommand{\halfms}{\half - s}
\newcommand{\thalfms}{\thalf - s}
\newcommand{\chibar}{\overline{\chi}}

\makeatletter
\@namedef{subjclassname@2020}{\textup{2020} Mathematics Subject Classification}
\makeatother

\numberwithin{equation}{section}

\newtheorem{theorem}{Theorem}

\newtheorem{lemma}[theorem]{Lemma}
\newtheorem{corollary}[theorem]{Corollary}

\begin{document}

\title{On moments of $L$-functions over Dirichlet characters}

\author{Avery Bainbridge}

\author{Rizwanur Khan}

\author{Ze Sen Tang}

\address{
Department of Mathematical Sciences\\ University of Texas at Dallas\\ Richardson, TX 75080-3021}
\email{Avery.Bainbridge@UTDallas.edu, Rizwanur.Khan@UTDallas.edu, ZeSen.Tang@UTDallas.edu}

\thanks{The second and third authors were supported by the National Science Foundation grants DMS-2344044 and DMS-2341239}

\begin{abstract}  We give a new proof of Heath-Brown's full asymptotic expansion for the second moment of Dirichlet $L$-functions and we obtain a corresponding asymptotic expansion for a twisted first moment of Hecke-Maass $L$-functions. 
\end{abstract}

\maketitle

\section{Introduction}

The Dirichlet $L$-functions are ubiquitous objects in analytic number theory. They arise in the study of primes in arithmetic progressions, class numbers of number fields, and many other problems. For a primitive Dirichlet character $\chi$ of modulus $q$, the associated Dirichlet $L$-function is defined initially for $\Re(s)>1$ by the absolutely convergent Dirichlet series
\[
L(s,\chi)=\sum_{n\ge 1} \frac{\chi(n)}{n^s}
\]
and then by analytic continuation to the rest of $\mathbb{C}$. The behavior of $L(s,\chi)$ within the critical strip $0\le s\le 1$ is the most important and mysterious. An important area of research is to try to understand these $L$-functions at the center of the critical strip by evaluating the moments
\[
\sum_{\chi\bmod q } |L(\thalf, \chi)|^{2k}
\]
for $k\in\mathbb{N}$, as $q\to\infty$. Unfortunately our knowledge of these moments is very limited, with $k=2$ being the largest value for which an asymptotic formula is known. We specialize to the important case of $q$ an odd prime to describe these results. Heath-Brown \cite{hb2} proved an asymptotic formula for the fourth moment ($k=2$), as follows:
\begin{align}
\label{hb4} \sum_{\chi\bmod q } |L(\thalf, \chi)|^{4} = \frac{q-1}{2\pi^2} \log^4 q +O(\log^3 q).
\end{align}
Subsequently there was great interest in refining this asymptotic. Young \cite{you} obtained an asymptotic formula which additionally includes all lower order main terms up to a much sharper error term: one that saves a power of $q$ over the main term, unlike \eqref{hb4} which only saves a power of $\log q$. This error term was then improved in other works, the best result \cite{bfkmm}  currently being
\[
\sum_{\chi\bmod q } |L(\thalf, \chi)|^{4} = (q-1)P_4(\log q)+O(q^{1-\frac{1}{20}+\epsilon})
\]
for some degree four polynomial $P_4(x)$ and any $\epsilon>0$. For the second moment ($k=1)$, Heath-Brown \cite{hb} proved an asymptotic formula with arbitrarily high precision. That is, he obtained all main terms up to an error term that saves as large a power of $q$ as desired. Since then, different proofs of Heath-Brown's asymptotic expansion (and generalizations in various directions) have been given in \cite{katmat, ega, bet, nord2}. Our first contribution is another proof of Heath-Brown's result, using the principle of analytic continuation and stationary phase analysis. We then use our method to obtain a new asymptotic expansion for a cuspidal analogue of the second moment of Dirichlet $L$-functions, which improves on Corollary 1.7 of Drappeau and Nordentoft \cite{dranor}.

\begin{theorem}\label{result} {{\cite[cf. Equation (5)]{hb}}} Let $q$ be an odd prime. For any integer $N\ge 1$ we have
\begin{align}
\label{resulteqn}
  &\sume_{\chi\bmod q } |L(\thalf, \chi)|^2 = \frac{q-1}{2} \Big( \log \frac{q}{8\pi} +  \gamma -\frac{\pi}{2}  \Big) + 2 \zeta^2(\thalf) q^\half  - 2\zeta^2(\thalf) + \sum_{n=1 }^{N} c_n q^{-\frac{n}{2}}+O_N\Big(q^{\frac{-(N+1)}{2}}\Big),\\ 
 \label{resulteqn-odd}   &\sumo_{\chi\bmod q } |L(\thalf, \chi)|^2 = \frac{q-1}{2} \Big( \log \frac{q}{8\pi} +  \gamma + \frac{\pi}{2}  \Big) +  \sum_{n=1 }^{N} c_n' q^{-\frac{n}{2}}+O_N\Big(q^{\frac{-(N+1)}{2}}\Big),
\end{align}
for some constants $c_n, c_n'$, where $\sume$ denotes the sum over the even primitive Dirichlet characters, $\sumo$ denotes the sum over the odd primitive Dirichlet characters, and $\gamma$ is Euler's constant.
\end{theorem} 
\noindent One may also include the trivial character in the first sum -- this would contribute $|\zeta(\half) (1-q^{-\half})|^2$.

Let $f$ a Hecke-Maass cusp form for $SL_2(\mathbb{Z})$ with Laplacian eigenvalue $\frac14+t_f^2$ and Hecke eigenvalues $\lambda_f(n)$. For $\chi$ a primitive Dirichlet character mod $q$, let $L(s,f\times \chi)$ denote the $L$-function of $f$ twisted by $\chi$. For $\Re(s)>1$, this is given by
\[
L(s,f\times \chi) = \sum_{n\ge 1} \frac{\lambda_f(n) \chi(n)}{n^s},
\]
and elsewhere by analytic continuation. We have
\begin{theorem} \label{result2} Let $f$ be a Hecke-Maass cusp form for $SL_2(\mathbb{Z})$ and let $q$ be an odd prime. For any integer $N\ge 1$ we have
\begin{align}
\label{resulteqn2}
&\sume_{\chi\bmod q } \frac{\tau(\chi)}{q^\half} L(\thalf, f\times \overline{\chi}) = 2 L(\thalf, f) q^\half  - \lambda_f(q) L(\thalf, f)  + \sum_{n=1 }^{N} c_{n,f} q^{-\frac{n}{2}}+O_{N,f}\Big(q^{\frac{-(N+1)}{2}}\Big),\\
\nonumber &\sumo_{\chi\bmod q } \frac{\tau(\chi)}{q^\half} L(\thalf, f\times \overline{\chi}) = \sum_{n=1 }^{N} c_{n,f}'  q^{-\frac{n}{2}}+O_{N,f}\Big(q^{\frac{-(N+1)}{2}}\Big),
\end{align}
for some constants $c_{n,f}, c_{n,f}'$, where where $\sume$ denotes the sum over the even primitive Dirichlet characters, $\sumo$ denotes the sum over the odd primitive Dirichlet characters, and $\tau(\chi)$ denotes the Gauss sum. 
\end{theorem} 
\noindent We have that $L(\half,f)=0$ when $f$ is an odd Hecke-Maass cusp form, and $\lambda_f(q)\ll q^{\frac{7}{64}+\epsilon}$ by \cite[Appendix 2]{kim}.

We will give a detailed proof of Theorem \ref{result}. The proof of Theorem \ref{result2} is very similar, so we will just sketch the differences to the proof of Theorem \ref{result}.

\subsection{Connections to reciprocity relations} The moments above can be generalized to twisted moments, which are found to obey certain reciprocity relations. This phenomenon was first observed by Conrey \cite{con}, who discovered an identity between the twisted moments 
\[
\sum\limits_{\chi\bmod q} |L(\thalf, \chi)|^2 \chi(p) \ \ \  \text{ and  } \sum\limits_{\chi\bmod p} |L(\thalf, \chi)|^2 \chi(-q),
\]
up to a modest error term. Conrey's result was improved by Young \cite{you2} and then extended to a full asymptotic series by Bettin \cite{con}, which gives Heath-Brown's asymptotic series as a special case. Similarly, reciprocity relations have been considered for the twisted moments
\[
\sum\limits_{\chi\bmod q} \tau(\chi) L(\thalf, f\times \overline{\chi}) \chi(p),
\]
where $f$ is either a holomorphic Hecke cusp form or non-holomorphic Hecke-Maass cusp form for $SL_2(\mathbb{Z})$ (or other congruence groups). Nordentoft \cite[Theorem 4.4]{nor} established an exact reciprocity formula in the holomorphic case, from which follows exact formulae for the moments considered in Theorem \ref{result2}, such that the series $\sum c_{n,f} q^{-n}$ has only finitely many terms and there is no error term. For Hecke-Maass cusp forms, Drappeau and Nordentoft \cite[Corollary 1.7]{dranor} have established a reciprocity formula up to a modest error term. This does not yield our Theorem \ref{result2} which allows for an arbitrarily strong error term. The proofs of the reciprocity formulae in \cite{bet, nor, dranor} can all be seen in the light of a special property known as quantum modularity for additive twists of the respective $L$-functions. This works out very elegantly especially in the case of holomorphic Hecke cusp forms, but is more complicated for Hecke-Maass cusp forms. On the other hand, we believe that our proof has the advantage of being simpler, because besides the analytic techniques, the only arithmetic input required is the functional equation and analytic continuation of the $L$-functions. Our method can also be extended to establish reciprocity relations for twisted moments. This will be addressed in forthcoming work of Agniva Dasgupta and the second and third authors.

\noindent

\section{Sketch}

In this section, we describe the main ideas of the proof of \eqref{resulteqn} in Theorem \ref{result}. Fix an odd prime $q$. Define for $s\in\mathbb{C}$ the entire function
\[
 F(s,q):=\sume_{\chi\bmod q} L(\thalfms, \chi) L(\thalfps, \chibar).
 \]
Our strategy is to derive a broader statement of the form
$F(s,q) = G(s,q)$,
where $G(s,q)$ is a function that is holomorphic on some open neighborhood of $s=0$, and then to recover Theorem \ref{result} as the specialization at $s=0$. To this end we begin by restricting $F(s,q)$ to $s\in \mathcal{S}_1$, an open set positioned far to the right in $\mathbb{C}$, as pictured in Figure \ref{figure}, and calling the restriction $H(s,q)$. After an application of the functional equation of $L(\thalfms, \chi)$, we obtain 
\begin{align}
\label{compare1} H(s,q) = \frac{\Gamma\left(\tfrac{1}{4}+\tfrac{s}{2}\right)}{\Gamma\left(\tfrac{1}{4}-\tfrac{s}{2}\right) \pi^s q^{\half-s}} \sume_{\chi\bmod q} \tau(\chi)   L(\thalfps,\chibar)^2.
\end{align}
Expanding the $L$-functions as absolutely convergent Dirichlet series, opening the Gauss sum, and executing the character sum using character orthogonality shows that $H(s,q)$ essentially equals
\[
\sum_{n\ge 1} \frac{\cos(\frac{2\pi n}{q})d(n)}{n^{\half+s}},
\]
where $d(n)$ is the divisor function. Writing $\cos(\frac{2\pi n}{q})$ as a Mellin transform, we recast this sum in terms of the integral 
\begin{align}
\label{compare2} \frac{1}{2\pi i}  \int\limits_{(-\frac34)}  \Gamma(w)\cos(\tfrac{\pi w}{2}) \Big(\frac{q}{2\pi}\Big)^w \zeta^2 (\thalf+s+w) \ dw.
\end{align}
We then move the line of integration far to the left, picking up residues that contribute to the main terms on the right hand side of \eqref{resulteqn}, with the leading main term arising from the double pole of $\zeta^2(\thalf+s+w)$ at $w=\thalf-s$. However, the shifted integral (now positioned far to the left at $\Re(w)=-c$ for $c>0$ large) does not converge absolutely near $s=0$, and so we need additional ideas to obtain meromorphic continuation near $s=0$. Keeping in mind that $w$ is far to the left in the shifted integral, we apply the functional equation of $\zeta^2(\thalf+s+w)$ to obtain $\zeta^2(\thalf-s-w)$, which we can expand into an absolutely convergent Dirichlet series since $-w$ is far to the right. Interchanging the order of integration and summation, we get 
\begin{align}
\label{sum-integral} \sum_{n\ge 1} \frac{d(n)}{n^{\half-s}}  \bigg(\frac{1}{2\pi i} \int\limits_{(-c)}   \Gamma(w)\cos(\tfrac{\pi w}{2}) \Big(\frac{nq}{2\pi}\Big)^w  \frac{\pi^{-\frac12+s+w} \Gamma^2(\frac14-\frac{s}{2}-\frac{w}{2}) }{\pi^{-\frac12-s-w} \Gamma^2(\frac14+\frac{s}{2}+\frac{w}{2})}  dw\bigg).
\end{align}
Evaluating the inner integral is a purely analytic problem (it's the inverse Mellin transform of a product of special functions). We approach this problem using Stirling's estimates and techniques of stationary phase analysis. We find that the integrand has oscillation given by 
\begin{align}
\label{osc-factor} \exp\left(-it \log \frac{|t|}{2\pi e n q}\right),
\end{align}
for large $|t|=|\Im(w)|$. There are stationary points at $t=\pm 2\pi n q$, and at these points the oscillatory factor \eqref{osc-factor} is independent of $n$. It is interesting to note that this step is special to the second moment, for such a simplification does not happen if we were trying to apply our method to evaluate higher moments.  It turns out that the leading term of the integral is essentially $(nq)^{-2s}$, which is especially simple, and from this the leading term of the sum \eqref{sum-integral} is 
\[
q^{-2s} \sum_{n\ge 1} \frac{d(n)}{n^{\half+s}}= q^{-2s} \zeta^2(\thalf+s).
\]
Thus we recover the Riemann Zeta function and we are able to use its meromorphic continuation from $\mathcal{S}_1$ to $\mathcal{S}_2$, and in particular evaluate at $s=0$. The lower order terms arising from the stationary phase analysis are of a similar shape, their values at $s=0$ decreasing by factors of $q^{-1}$, and thus contribute to the lower order main terms on the right hand side of \eqref{resulteqn}. We denote by $G(s,q)$ the sum of all these terms plus the residues obtained above, and get that $F(0,q)=G(0,q)$.

\begin{figure}
    \begin{tikzpicture}[scale=0.85]
        \draw[-latex] (-6.2,0) -- (6.7,0) node[right] {$\Re(s)$};
        \draw[-latex] (0,-2.2) -- (0,2.2) node[above] {$\Im(s)$};
        \coordinate [label=below left:0] (a) at (0,0);
        \draw[red] (4,0.5) -- (5,0.5);
        \draw[red] (4,-0.5) -- (5,-0.5);
        \draw[red] (4,0.5) -- (4,-0.5);
        \draw[red] (5,0.5) -- (5,-0.5);
        \draw[blue] (-0.5,0.5) -- (5,0.5);
        \draw[blue] (-0.5,-0.5) -- (5,-0.5);
        \draw[blue] (-0.5,0.5) -- (-0.5,-0.5);
        \draw[blue] (5,0.5) -- (5,-0.5);
        \node[red] at (4.5,0.2) {$\mathcal{S}_1$};
        \node[blue] at (2,0.2) {$\mathcal{S}_2$};
    \end{tikzpicture}
    \caption{} \label{figure}
\end{figure}

For the proof of \eqref{resulteqn2} in Theorem \ref{result2}, we observe that the left hand side is the value at $s=0$ of
\[
q^{-\half}\sume_{\chi\bmod q} \tau(\chi)   L(\thalfps,f\times \chibar),
\]
which is analogous to \eqref{compare1}. We take the Dirichlet series expansion of $L(\thalfps,f\times \chibar)=\sum_{n\ge 1} \frac{\lambda_f(n)\overline{\chi}(n)}{n^{\half+s}}$ for $s\in\mathcal{S}_1$, which is analogous to $L(\thalfps, \chibar)^2=\sum_{n\ge 1} \frac{d(n)\overline{\chi}(n)}{n^{\half+s}}$, except that $d(n)$ is replaced with $\lambda_f(n)$. The proof proceeds as for Theorem \ref{result}, except for a few changes. In \eqref{compare2}, we get the function $L(\half +s+w,f)$, which is entire on $\mathbb{C}$, instead of $\zeta^2(\half +s +w)$.  As a result, when we shift contours to the left, there is no pole at $s=\half-w$, unlike the case was for $\zeta^2(\half +s+w)$, and therefore the leading main term present in \eqref{resulteqn} does not arise. When the functional equation \cite[Proposition 3.13.5]{gol} is applied to $L(\half +s+w,f)$, we obtain the ratio of Gamma functions 
\[
(-1)^{\kappa_f} \frac{\Gamma(\frac14+ \frac{\kappa_f}{2} + \frac{it_f}{2} -\frac{s}{2}-\frac{w}{2}  )\Gamma(\frac14 + \frac{\kappa_f}{2}  - \frac{it_f}{2}  -\frac{s}{2}-\frac{w}{2} ) }{  \Gamma(\frac14 + \frac{\kappa_f}{2}  + \frac{it_f}{2} + \frac{s}{2}+\frac{w}{2}  )\Gamma(\frac14 + \frac{\kappa_f}{2}  - \frac{it_f}{2} +\frac{s}{2}+\frac{w}{2} )  } 
\]
 instead of what is seen in \eqref{sum-integral}, where $\kappa_f= 0$ if $f$ is an even Hecke-Maass form and  $\kappa_f= 1$ if $f$ is an odd Hecke-Maass form. However this difference is immaterial as the oscillatory behavior \eqref{osc-factor} remains the same after Stirling's estimates are applied. In the end, we do not get main terms of size $q\log q$ and $q$ in \eqref{resulteqn2}, and we have the lower order main term $ L(\half, f) (2q^{\half} -\lambda_f(q))$  instead of the $\zeta^2(\half) (2q^{\half} -d(q))$ term that is present in \eqref{resulteqn}.
 
\section{Preliminaries}
We recall here fundamental facts about Dirichlet characters, Dirichlet $L$-functions, and the Mellin transform. Throughout, $q$ is an odd prime. Recall that a Dirichlet character $\chi$ of modulus $q$ satisfies $\chi(-1)=\pm 1$, and it is called even if $\chi(-1)=1$ and odd otherwise. Since $q$ is prime, $\chi$ is primitive if and only if $\chi$ is not the principal character (which satisfies $\chi(n)=1$ for all $n$ coprime to $q$). We write $e(x)$ for the additive character $e^{2\pi i x}$.
\begin{lemma}[Functional equation]\label{dl-funceqn}
    Let $\chi$ be a primitive Dirichlet character of modulus $q$. Let $\delta=0$ if $\chi$ is even and $\delta=1$ if $\chi$ is odd. Let $ \tau(\chi)=\sum\limits_{k=1}^{q-1} e(\tfrac{k}{q}) \chi(k)$
    denote the Gauss sum. Then for all $s \in \mathbb{C}$ we have
    \[L(s,\chi) =\frac{\tau(\chi)}{i^\delta} \frac{\pi^{s-\half}}{q^s} \frac{\Gamma(\tfrac{1-s+\delta}{2})}{\Gamma(\tfrac{s+\delta}{2})} L(1-s,\chibar).\]
    For the Riemann Zeta function, we have
        \begin{align}
    \label{fe-2}   \zeta(s) =  \pi^{s-\half} \frac{\Gamma(\tfrac{1-s}{2})}{\Gamma(\tfrac{s}{2})} \zeta(1-s) = \pi^s 2^{s-1} \cos(\tfrac{\pi}{2}(1-s)) \Gamma(1-s) \zeta(1-s).
        \end{align}
\end{lemma}
\noindent The equation \eqref{fe-2}  uses the reflection formula $\Gamma(1-z)\Gamma(z)\sin(\pi z) = \pi$ and duplication formula $2^{2z-1} \Gamma(z) \Gamma(z+\thalf)=\sqrt{\pi} \Gamma(2z).$
\begin{lemma} [Orthogonality relations] \label{orthog-1}
    Suppose $(nm,q) = 1$. Let $\sume$ (respectively $\sumo$) denote the sum over the even (respectively odd) primitive Dirichlet characters. Then
    \[
    \sume_{\chi\bmod q} \chi(n)\chibar(m) = \begin{cases}\frac{q-3}{2}&\textnormal{if } n \equiv \pm m \bmod q, \\-1&\textnormal{otherwise,}\end{cases}, \ \ \ 
    \sumo_{\chi\bmod q} \chi(n)\chibar(m) = \begin{cases} \pm \frac{q-1}{2}&\textnormal{if } n \equiv \pm m \bmod q, \\
0 &\textnormal{otherwise.}\end{cases}\]
 \end{lemma}
\begin{proof}
The Dirichlet characters of modulus $q$ satisfy the following orthogonality property:
    \[\sum_{\chi\bmod q} \chi(n)\chibar(m) = \begin{cases} q-1&\textnormal{if } n \equiv m \bmod q,\\0&\textnormal{otherwise.}\end{cases}\]
    The result for the sum over even primitive characters then follows from the equation
    \[
    \sume_{\chi\bmod q} \chi(n)\chibar(m) =
  -1+  \sum_{\substack{\chi\bmod q\\ \chi(-1)=1}} \chi(n)\chibar(m) 
    = -1 + \sum_{\chi\bmod q} \left(\frac{1 + \chi(-1)}{2}\right) \chi(n)\chibar(m), 
    \]
    where the $-1$ term serves to subtract the contribution of the trivial character and the $\frac{1 + \chi(-1)}{2}$ factor serves to pick out the condition $\chi(-1)=1$ for the primitive characters. The result for the sum over odd primitive characters follows from the equation
    \[
    \sumo_{\chi\bmod q} \chi(n)\chibar(m) =
  \sum_{\substack{\chi\bmod q\\ \chi(-1)=-1}} \chi(n)\chibar(m) 
    =  \sum_{\chi\bmod q} \left(\frac{1 - \chi(-1)}{2}\right) \chi(n)\chibar(m). \qedhere
    \]
\end{proof}

\begin{corollary} \label{orthog-euler}
    Suppose $(n,q) = 1$. Then 
    \begin{align*}
     \sume_{\chi\bmod q} \tau(\chi)\overline{\chi}(n) =  (q-1) \cos(\tfrac{2 \pi n}{q}) + 1,  \ \ \ \ \ \ 
        \sumo_{\chi\bmod q} \frac{\tau(\chi)}{i} \overline{\chi}(n) =  (q-1) \sin(\tfrac{2 \pi n}{q}).
    \end{align*}
\end{corollary}
\begin{proof}
We only give the proof for the sum over the even primitive characters (as the proof for the sum over the odd primitive characters is similar). We have
\[
\sume_{\chi\bmod q} \tau(\chi)\overline{\chi}(n) = \sum_{k=1}^{q-1} e\Big(\frac{k}{q}\Big) \sume_{\chi\bmod q} \chi(k) \overline{\chi}(n) = \sum_{k=1}^{q-1} e\Big(\frac{kn}{q}\Big) \sume_{\chi\bmod q} \chi(kn) \overline{\chi}(n).
\]
We separate out the first $(k=1)$ and last ($k=q-1\equiv -1 \bmod q$) terms of the $k$-sum, to get
    \begin{align*}
        \sum_{k=1}^{q-1} e(\tfrac{kn}{q}) &\sume_{\chi\bmod q} \chi(kn)\chibar(n) \\
        = e(\tfrac{n}{q}) &\sume_{\chi\bmod q} \chi(n)\chibar(n) + e(\tfrac{-n}{q}) \sume_{\chi\bmod q} \chi(-n)\chibar(n) + \sum_{k=2}^{q-2} e(\tfrac{kn}{q}) \sume_{\chi\bmod q} \chi(kn)\chibar(n).
    \end{align*}
    Now we see that each sum in the last line corresponds to exactly one of the cases in the previous lemma, so executing the character sums yields
    \[\sum_{k=1}^{q-1} e(\tfrac{kn}{q}) \sume_{\chi\bmod q} \chi(kn)\chibar(n)  = \frac{q-3}{2} \left( e(\tfrac{n}{q}) + e(-\tfrac{n}{q}) \right) - \sum_{k=2}^{q-2} e(\tfrac{kn}{q})\]
    and we note that
    \[\sum_{k=2}^{q-2} e(\tfrac{kn}{q}) = -\left( e(\tfrac{n}{q}) + e(-\tfrac{n}{q}) \right) + \sum_{k=1}^{q-1} e(\tfrac{kn}{q}) = -\left( e(\tfrac{n}{q}) + e(-\tfrac{n}{q}) \right) - 1.\qedhere \]
\end{proof}

\begin{lemma}[Mellin inversion] \cite[Chapter III, Theorem 2]{rw} \label{mellinthm}
    Suppose that $f: (0, \infty) \to \mathbb{R}$ is continuous and the integral
    \begin{align*}
   \phi(w) = \int_0^\infty x^{w-1} f(x) dx
    \end{align*}
    is absolutely convergent for $a < \Re(w) < b$. Then
    \[f(x) = \frac{1}{2\pi i} \int_{(c)} x^{-w} \phi(w) dw\]
    for $c \in (a,b)$.
\end{lemma}

\begin{lemma}[An integral representation] \label{mellinrep} Let $x > 0$. For $d \in (-1,0)$,  we have
\begin{align}
 \label{mellinrep-mainline}\cos(x) = \frac{1}{2\pi i}\int_{(d)}  \Gamma(w) \cos\left(\frac{\pi w}{2}\right) x^{-w} dw + 1.
 \end{align}
    Further, this representation is absolutely convergent if $d \in (-1,-\half)$.
\end{lemma}
\begin{proof}
Let $f(x) = \cos(x) - e^{-x}$ and define
\[\tilde{f}(w) := \int_0^\infty x^{w-1} f(x) dx\]
for $-1<\Re(w)<1$. This is well defined because we have $f(x)= x + O(x^2)$ for $|x|\le \pi$ and $f(x)=\cos(x) + O(e^{-x})$ for $|x|> \pi$, so that $\tilde{f}(w)$ converges absolutely for $-1< \Re(w) <0$ and converges conditionally for $0\le \Re(w) <1$ (by integration by parts, as in \eqref{int-par}). We claim moreover that $\tilde{f}(w)$ is holomorphic for $-1<\Re(w)<1$. For $-1<\Re(w)<0$,  we can differentiate under the integral sign by absolute convergence. For $0\le \Re(w) <1$, we have that $\tilde{f}(w)$ equals the sum of a holomorphic function and $\int_\pi^\infty x^{w-1} \cos x \ dx$. But by integration by parts,
\begin{align}
\label{int-par} \int_\pi^\infty x^{w-1} \cos x \ dx = \int_\pi^\infty (w-1)x^{w-2} \sin x \ dx, 
\end{align}
which is absolutely convergent, and hence analytic. 

 For $0<\Re(w)<1$, we have by \cite[equations 17.43.1, 17.43.3]{gr} that
\[
\tilde{f}(w) = \int_0^\infty x^{w-1} \cos(x)  dx - \int_0^\infty x^{w-1} e^{-x}  dx = \Gamma(w)\cos\Big(\frac{\pi w}{2}\Big) - \Gamma(w).
\]
By analytic continuation, we get that \[\tilde{f}(w) = \Gamma(w)\cos\Big(\frac{\pi w}{2}\Big) - \Gamma(w)\] for $-1<\Re(w)<1$. Then by Lemma \ref{mellinthm}, we have
\begin{align}
\label{mel-inv-f} f(x) =\cos(x) - e^{-x} = \frac{1}{2\pi i} \int_{(d)}  \Gamma(w) \left(\cos\left(\frac{\pi w}{2}\right) - 1 \right) x^{-w} dw
\end{align}
for $d \in (-1,0)$. Also by Mellin inversion, since $\int_0^\infty e^{-x} x^{w-1} dx$ converges absolutely for $\Re(w) >0$, we have
\begin{align*}
  e^{-x} = \frac{1}{2\pi i} \int_{(d)}  \Gamma(w) x^{-w} dw
\end{align*}
for $d\in(0,1)$. We move the line of integration to the left, crossing a simple pole of $\Gamma(w)$ at $w=0$, to obtain
\begin{align}
\label{mel-inv-e} e^{-x} = \frac{1}{2\pi i} \int_{(d)}  \Gamma(w) x^{-w} dw + 1
\end{align}
for $d \in (-1,0)$. 
Adding together \eqref{mel-inv-f} and \eqref{mel-inv-e} gives \eqref{mellinrep-mainline}.

Finally, we have by Stirling's approximation (equation \eqref{gamma}) that 
\[
 \Gamma(d+it) \cos\left(\frac{\pi (d+it)}{2}\right)\ll |t|^{d-\half}. 
\]
Thus the integral in \eqref{mellinrep-mainline} converges absolutely for $d\in(-1,-\half)$.
\end{proof}

\begin{lemma}[A residue calculation] \label{residues}
    Define
    \[f(s,w,q) := \Gamma(w) \cos\left(\tfrac{\pi w}{2}\right)  \left(\frac{q}{2\pi}\right)^w\zeta^2(\thalfps+w)\] 
    for $s \in \mathbb{C}$ such that $\thalfms \not\in \mathbb{Z}$. Then $f$ has simple poles at $w = -2n$ for $n \in \mathbb{N}\cup \{0\}$, and a double pole at $w = \thalfms$. The residues at these poles are
    \[R_0(n,s,q):=\underset{w=-2n}{\mathrm{Res}} f(s,w,q) = \frac{(-1)^n}{(2n)!} \left(\frac{2\pi}{q}\right)^{2n} \zeta^2(\thalfps-2n)\]
    and
    \begin{align} 
      \nonumber &R_1(s,q):=\underset{w=\half-s}{\mathrm{Res}} f(s,w,q) \\
\label{r1res}       &=  \Gamma(\thalfms) \cos\left(\tfrac{\pi }{4}-\tfrac{\pi s }{2}\right)  \left(\tfrac{q}{2\pi}\right)^{\half-\frac{s}{2}} \Big( \tfrac{ \Gamma'(\halfms)}{ \Gamma(\halfms)} - \tfrac{\pi}{2} \tan\left(\tfrac{\pi }{4}-\tfrac{\pi s }{2}\right) + \log \left(\tfrac{q}{2\pi}\right) + 2\gamma \Big).
            \end{align}
             These residues are meromorphic functions of $s$ which are analytic at $s=0$.
\end{lemma}
\begin{proof}
    The assumption that $\thalfms \not\in \mathbb{Z}$ ensures that the poles under discussion do not coincide. Recall that $\Gamma(w)$ has simple poles at the non-positive integers, but $\cos(\tfrac{\pi w}{2})$ vanishes at the negative odd integers, so $\Gamma(w) \cos(\tfrac{\pi w}{2})$ only has simple poles at the non-positive even integers $w=-2n$. We have the Laurent expansion 
 \[\Gamma(w)=\frac{1}{(2n)!}(w+2n)^{-1}+\ldots\] about $w=-2n$, from which we get that 
\[
\underset{w=-2n}{\mathrm{Res}} g(w) \Gamma(w) = \frac{g(-2n)}{(2n)!}
\]
for any function $g(w)$ that is holomorphic in a neighborhood of $w=-2n$. This leads to the calculation of $R_0(n,s,q)$.

Recall that $\zeta^2(\thalfps+w)$ has a double pole at $w=\thalfms$, and $\cos(\frac{\pi(\half-s)}{2}) \neq 0$ as $\thalfms \not\in \mathbb{Z}$. We have the Laurent expansion 
\[
\zeta^2(\thalf+s+w)=(w-\thalf+s)^{-2}+2\gamma(w-\thalf+s)^{-1}+\ldots
\]
about $w=\half-s$, where $\gamma$ is Euler's constant. From this we have that if $g(w)$ is holomorphic in a neighborhood of $w=\half-s$, then 
\[
\underset{w=-2n}{\mathrm{Res}} g(w)\zeta^2(\thalfps+w) =g'(\thalf-s)+2\gamma g(\thalf-s). \]
This leads to the calculation of $R_1(s,q)$.
\end{proof}
\noindent For future reference, using the values $\Gamma(\half)=\pi^\half, \frac{\Gamma'(\half)}{\Gamma(\half)}=\log(\frac14)-\gamma, \cos(\frac{\pi}{4})=2^{-\half}, \tan(\frac{\pi}{4}) =1 $, we note that
\begin{align}
\label{record} R_1(0,q)=\frac{q^{\half}}{2} \Big( \log \left(\tfrac{q}{8\pi}\right)+  \gamma -\tfrac{\pi}{2} \Big).
\end{align}

\begin{lemma}[The convexity bound for the Riemann Zeta function] \cite[Chapter V]{tit} \label{growth}
Let $\sigma \in \mathbb{R}$ and $\epsilon>0$. For $t\in\mathbb{R}$ with $|t|\ge 1$, we have 
\begin{align}
\label{convexitybound} |\zeta(\sigma + it)| \ll |t|^{\alpha+\epsilon},
\end{align}
where $\alpha=0$ for $\sigma>1$, $\alpha=\half (1 - \sigma)$ for $\sigma \in [0,1]$, and $\alpha= \thalf - \sigma$ for $\sigma<0$. The implied constant depends on $\sigma$ and $\epsilon$.
    \end{lemma}
    \noindent The result above will be used when bounding vertical integrals on fixed lines. But when these lines are shifted left or right in order to apply Cauchy's residue theorem, the argument will implicitly involve integrals over horizontal lines. Bounding such horizontal integrals will require some control over the implicit constant in \eqref{convexitybound}, and this may be found in \cite[Theorem 4]{rad}.

    \begin{lemma}[Stirling's expansion]\label{stirling}
Let $\mathcal{C}\subset \{z: \Re(z)>0\} $ be a fixed compact set. For $z\in\mathcal{C}$ and $t\in\mathbb{R}$ with $|t|>\half$, we have 
\begin{align}
\label{gamcos} \Gamma(z+it)\cos(\tfrac{\pi}{2}(z+it)) =  |t|^{(z-\half)} \exp\Big( it \log|t| -it  \Big)\Big(\sum_{m=0}^M a_m(z) t^{-m} + E_M(z,t) \Big)
\end{align}
for any $M\ge 0$, where $a_m(z)$ and $E_M(z,t)$ are holomorphic functions of $z\in\mathcal{C}$ that depend on 
$\mathrm{sgn}(t)$ such that $a_m(z)$ is a polynomial in $z$ of degree at most $2m$, 
\begin{align}
\label{a0val} a_0(z) = (\tfrac{\pi}{2})^\half e^{-i \frac{\pi}{4}\mathrm{sgn}(t) },
\end{align}
\and
\begin{align}
\label{amprop} E_M(z,t) \ll_{M} |t|^{-M-1}.
\end{align}
\end{lemma}
\proof
We work with the principal branch of the complex logarithm. Throughout the proof, $M$ will denote a natural number, $a_m(z)$ and $E_M(z,t)$ will denote functions as described in the lemma, but all of these quantities may not be the same from one occurrence to the next, and $a_m$ will denote some constants. By \cite[equation 8.344]{gr}, we have
\begin{align}
\label{quote-stirling} \log \Gamma(z+it)=(z-\thalf+it)\log(z+it)-(z+it) + \thalf \log 2\pi+\sum_{r=1}^{R} \frac{a_r}{(z+it)^r} + E_R(z,t).
\end{align}
Although \cite[equation 8.344]{gr} just gives a bound for $E_R(z,t)$, we can infer that it is holomorphic for $z\in\mathcal{C}$ since every other term in \eqref{quote-stirling} is. Using Taylor's theorem, we write
\begin{align*}
 \log (z+it) = \log(it)+\log \Big(1+\frac{z}{it}\Big) = \log|t| +i \frac{\pi}{2} \mathrm{sgn}(t) + \sum_{m=1}^M \frac{(-1)^{m+1}}{m}\Big(\frac{z}{it}\Big)^m + E_M(z,t)
\end{align*}
and 
\[
\frac{a_r}{(z+it)^r}=\frac{a_r}{(it)^r} \frac{1}{(1+\frac{z}{it})^r}= \frac{a_r}{(it)^r}\bigg( 1+ \sum_{m=1}^M
\mybinom{-r}{m}\Big(\frac{z}{it}\Big)^m+ E_M(z,t) \bigg).
\]
Plugging back into \eqref{quote-stirling}, the $-z$ term cancels with $it \cdot \frac{z}{it}$, and we get 
\begin{align*}
 \log \Gamma(z+it)=(z-\thalf+it) \log |t| - \tfrac{\pi}{2} |t| -it   + i(z-\thalf) \tfrac{\pi}{2} \mathrm{sgn}(t) +\thalf\log 2\pi+ \sum_{m=1}^M a_m(z) t^{-m} + E_M(z,t).
\end{align*}
Taking the exponential of both sides we get
\begin{align*}
&\Gamma(z+it)\\
&=\sqrt{2\pi} e^{-i \frac{\pi}{4} \mathrm{sgn}(t) } |t|^{(z-\half)} \exp\Big( -\tfrac{\pi}{2}|t|   +it\log|t| -it + iz \tfrac{\pi}{2} \mathrm{sgn}(t) \Big) \prod_{m=1}^M \exp(a_m(z) t^{-m} ) \cdot \exp(E_M(z,t)).
\end{align*}
Next we take a Taylor polynomial for $\exp(x)$ to get 
\begin{align}
\label{gamma} &\Gamma(z+it) \\
\nonumber &=\sqrt{2\pi} e^{-i \frac{\pi}{4} \mathrm{sgn}(t) } |t|^{(z-\half)} \exp\Big( -\tfrac{\pi}{2}|t|   +it\log|t| -it + iz \tfrac{\pi}{2} \mathrm{sgn}(t) \Big)\Big(1+\sum_{m=1}^M a_m(z) t^{-m} + E_M(z,t) \Big).
\end{align}
To obtain \eqref{gamcos}, we write
\begin{align}
\label{cos} \cos(\tfrac{\pi}{2}(z+it)) &= \thalf\big( e^{i(\frac{\pi}{2}(z+it))}+e^{-i(\frac{\pi}{2}(z+it))}\big)=\thalf e^{-i  z\tfrac{\pi}{2} \mathrm{sgn}(t) + \tfrac{\pi}{2} |t|} \big(1 + e^{iz \pi   \mathrm{sgn}(t)  -\pi |t|} \big)
\end{align}
and then multiply with \eqref{gamma}. 
\endproof


\section{Obtaining the leading main term in Theorem \ref{result}}

We focus on \eqref{resulteqn}. Fix any constant $c \in \thalf + \mathbb{Z}$ satisfying $c > 10$. This parameter determines the length of our asymptotic expansion: we will see that the larger $c$ is, the more terms we may take on the right hand side of \eqref{resulteqn}. Define the sets
\begin{align*}
&\mathcal{S}_1 := \{s \in \mathbb{C} \,:\, -1 < \Im(s) < 1, \ c-\tfrac32 < \Re(s) < c-\thalf\},\\
&\mathcal{S}_2 := \{s \in \mathbb{C} \,:\, -1 < \Im(s) < 1, \ -\tfrac{1}{10} < \Re(s)< c-\thalf\}.
\end{align*}

Define for $s \in \mathbb{C}$ and an odd prime $q$,
\begin{align*}
    F(s,q) := \sume_{\chi\bmod q} L(\thalfms, \chi) L(\thalfps, \chibar)
\end{align*}
where the sum is taken over the even primitive Dirichlet characters of modulus $q$. Let $H(s,q)$ denote the restriction of $F(s,q)$ to the set $\mathcal{S}_1$.
By the functional equation given in Lemma \ref{dl-funceqn}, we have that
\begin{align*}
    H(s,q) = \frac{\Gamma\left(\tfrac{1}{4}+\tfrac{s}{2}\right)}{\Gamma\left(\tfrac{1}{4}-\tfrac{s}{2}\right) \pi^s q^{\half-s}} \sume_{\chi\bmod q} \tau(\chi)   L(\thalfps,\chibar)^2 := \frac{\Gamma\left(\tfrac{1}{4}+\tfrac{s}{2}\right)}{\Gamma\left(\tfrac{1}{4}-\tfrac{s}{2}\right) \pi^s q^{\half-s}} H_1(s,q).
\end{align*}
Now, since in $\mathcal{S}_1$ we have $\Re(s) > \thalf$, the Dirichlet series representation of $L(\thalfps, \chibar)$ is absolutely convergent, and hence it follows that
\begin{align*}
    H_1(s,q) =  \sume_{\chi\bmod q}  \ \sum_{n,m \geq 1} \frac{\tau(\chi) \chibar(nm)}{(nm)^{\half+s}} = \sum_{\substack{n\ge 1\\(n,q)=1}}  \ \sume_{\chi\bmod q} \frac{\tau(\chi) d(n) \chibar(n) }{n^{\half+s}} .
\end{align*}
The last equality follows by exchanging the order of summation, using the total multiplicativity of the Dirichlet characters, and using the fact that $\chi(n)=0$ if $q|n$. We apply Corollary \ref{orthog-euler} to get
\begin{align*}
    H_1(s,q) =  \sum_{\substack{n \geq 1\\ (n,q)=1}}   \frac{((q-1) \cos(\tfrac{2 \pi n}{q}) + 1)d(n)}{n^{\half+s}}.
\end{align*}
The sum may be extended to all natural numbers as follows:
\begin{align*}
H_1(s,q) 
&=  \sum_{\substack{n \geq 1}}   \frac{((q-1) \cos(\tfrac{2 \pi n}{q}) + 1)d(n)}{n^{\half+s}}- \sum_{\substack{n \geq 1\\ q|n}}   \frac{((q-1) \cos(\tfrac{2 \pi n}{q}) + 1)d(n)}{n^{\half+s}}\\
&=  \sum_{\substack{n \geq 1}}   \frac{((q-1) \cos(\tfrac{2 \pi n}{q}) + 1)d(n)}{n^{\half+s}}- q\sum_{\substack{m \geq 1}}   \frac{ d(mq)}{(mq)^{\half+s}},
   \end{align*}
where we wrote $n=qm$ to get the last expression. By the Hecke multiplicative property of the divisor function (which is also satisfied by $\lambda_f(n)$),
\[
d(mq)=
\begin{cases}
d(m)d(q)  &\text{ if } (m,q)=1,\\
d(m)d(q)-d(\frac{m}{q})  &\text{ if } q|m,
\end{cases}
\]
we get
\begin{align}
\label{passtoh2}  H_1(s,q) &= (q-1) \sum_{\substack{n \geq 1}}   \frac{ \cos(\tfrac{2 \pi n}{q}) d(n)}{n^{\half+s}} + (1 + q^{-2s} - d(q) q^{\half-s}) \zeta^2(\thalfps)\\
\nonumber &:= (q-1)H_2(s,q)+ (1 + q^{-2s} - 2q^{\half-s}) \zeta^2(\thalfps),
\end{align}
where we used $d(q)=2$ in the last line. Note that for Theorem \ref{result2}, we get $\lambda_f(q)$ instead of  $d(q)$.

Inserting the integral representation of $\cos(\tfrac{2\pi n}{q})$ given in Lemma \ref{mellinrep}, we have
\[
H_2(s,q) = \sum_{n\ge 1} \frac{d(n)}{n^{\half+s}} \frac{1}{2\pi i} \int_{(-\frac{3}{4})} \Gamma(w) \cos\left(\tfrac{\pi w}{2}\right)  \left(\frac{q}{2\pi n}\right)^w dw + \zeta^2(\thalf+s).
\]
The integral converges absolutely:
\begin{equation}
   \int_{(-\frac34)}  \left| \Gamma(w) \cos\left(\frac{\pi w}{2}\right) \frac{q^w}{(2\pi n)^w} dw\right| \ll \frac{n^{\frac34}}{q^{\frac34}},
\end{equation}
and $\Re(\thalf + s - \tfrac{3}{4}) > 1$ for $s\in\mathcal{S}_1$, so we can interchange summation and integration to get
\begin{align}
 \label{h2interchange}   H_2(s,q) &= \frac{1}{2\pi i} \int_{(-\frac{3}{4})} \Gamma(w) \cos\left(\tfrac{\pi w}{2}\right)  \left(\frac{q}{2\pi}\right)^w\zeta^2(\thalfps+w) dw + \zeta^2(\thalfps) \\
  \nonumber  &:= H_3(s,q) + \zeta^2(\thalfps).
\end{align}
Moving the line of integration to $\Re(w)=-c$, we pick up the residues calculated in Lemma \ref{residues}, and get
\[ H_3(s,q)=  R_1(s,q) + \sum_{n=1}^{\lfloor c/2 \rfloor} R_0(n,s,q) + I(s,q),\]
where $\lfloor x\rfloor$ denotes the greatest integer less than or equal to $x$, and 
\[I(s,q) = \frac{1}{2\pi i} \int_{(-c)} \Gamma(w) \cos\left(\tfrac{\pi w}{2}\right)  \left(\frac{q}{2\pi}\right)^w\zeta^2(\thalfps+w) dw.\]
Note that this integral converges absolutely, since by Lemma \ref{growth} we have
\[|\Gamma(w)\cos\left(\tfrac{\pi w}{2}\right)\zeta^2(\thalfps+w)| \ll |t|^{-c-\half} |t|^{1-2\Re(\half+s+w)+\epsilon}  \ll |t|^{3} |t|^{-c-\half} \ll |t|^{-7}\]
for $w = -c + it$ with $|t| \to \infty$.

Retracing our steps and inserting the final expressions for $H_1(s,q),H_2(s,q),H_3(s,q)$ back into $H(s,q)$, we obtain the formula
\begin{multline}
\label{Hresult} H(s,q)=\frac{\Gamma(\tfrac{1}{4}+\tfrac{s}{2})}{\Gamma(\tfrac{1}{4}-\tfrac{s}{2}) \pi^s q^{\half-s}} \bigg((q-1)R_1(s,q) + (q-1)\sum_{n=1}^{\lfloor c/2 \rfloor} R_0(n,s,q) \\ + (q-1)I(s,q) + (q + q^{-2s} - 2q^{\half-s}) \zeta^2(\thalfps)\bigg).
\end{multline}
From this expression, it follows that 
\[H(s,q)-\frac{\Gamma\left(\tfrac{1}{4}+\tfrac{s}{2}\right)(q-1)}{\Gamma\left(\tfrac{1}{4}-\tfrac{s}{2}\right) \pi^s q^{\half-s}}I(s,q),\]
continues to a meromorphic function on $\mathcal{S}_2$, with value at $s=0$ equal to
\begin{align}
\label{eval-at-0} \frac{q-1}{2} \left( \gamma +\log(\tfrac{q}{8\pi}) - \tfrac{\pi}{2} \right) + \frac{(q-1)}{q^{\frac{1}{2}}}\sum_{n=1}^{\lfloor c/2 \rfloor} \frac{(-1)^n}{(2n)!}\Big(\frac{2\pi}{q}\Big)^{2n} \zeta^2(\thalf-2n)   +\frac{(q +1- 2q^{\frac{1}{2}})}{q^{\frac{1}{2}}} \zeta^2(\thalf),
\end{align}
using \eqref{record}. We see here some of the main terms claimed in equation \eqref{resulteqn} of Theorem \ref{result}: namely, the leading main term $ \frac{q-1}{2} ( \log(\frac{q}{8\pi})+\gamma - \frac{\pi}{2} )$, half of the second main term $2  \zeta^2(\half) q^{\half}$, and  the third main term $-2\zeta^2(\half)$, while the rest of the right hand side of \eqref{eval-at-0} is $O(q^{-\half})$ and contributes to the lower order main terms in Theorem \ref{result}. However, it still remains to study the shifted integral $I(s,q)$, and this is the purpose of the next section. 

We will show in the next section that as a function of $s$, 
\[
\frac{\Gamma(\frac{1}{4}+\frac{s}{2})(q-1)}{\Gamma(\frac{1}{4}-\frac{s}{2}) \pi^s q^{\half-s}}I(s,q)
\]
continues to a meromorphic function on $\mathcal{S}_2$, with value at $s=0$ equal to  $\zeta^2(\half) q^{\half}$ plus addition lower order main terms of size $O(q^{-\half})$. This recovers the other half of the second main term $2  \zeta^2(\half) q^{\half}$ in \eqref{resulteqn}.

 To frame this in terms of the strategy laid out in the sketch, we let $\mathcal{I}(s,q)$ denote meromorphic continuation of $I(s,q)$ to $\mathcal{S}_2$. Then 
\begin{multline*}
G(s,q):=\frac{\Gamma(\tfrac{1}{4}+\tfrac{s}{2})}{\Gamma(\tfrac{1}{4}-\tfrac{s}{2}) \pi^s q^{\half-s}} \bigg((q-1)R_1(s,q) + (q-1)\sum_{n=1}^{\lfloor c/2 \rfloor} R_0(n,s,q) \\ + (q-1)\mathcal{I}(s,q) + (q + q^{-2s} - 2q^{\half-s}) \zeta^2(\thalfps)\bigg)
\end{multline*}
is meromorphic for $s\in\mathcal{S}_2$, with $G(s,q)=H(s,q)$ on $\mathcal{S}_1$. Since also $F(s,q)=H(s,q)$ on $\mathcal{S}_1$, we get that $F(s,q)=G(s,q)$ on $\mathcal{S}_2$ by the principle of analytic continuation.

\

{\it The sum over the odd characters.} Before moving on we note that deriving \eqref{resulteqn-odd} would entail similar calculations, except for a couple of important differences as follows. Our main object of study would be
\[
 (q-1) \sum_{\substack{n \geq 1}}   \frac{ \sin(\tfrac{2 \pi n}{q}) d(n)}{n^{\half+s}}.
 \]
Initially this sum would be restricted to $(n,q)=1$ but would extend to all $n\ge 1$ for free since $\sin(\frac{2 \pi n}{q})=0$ for $q|n$. Thus there would be no additional terms as there are in \eqref{passtoh2}. We would use the Mellin inverse transform
\begin{align}
\label{inv-mel-sin} \sin(x) = \frac{1}{2\pi i}\int_{(d)}  \Gamma(w) \sin\left(\frac{\pi w}{2}\right) x^{-w} dw
\end{align}
for $d\in(-1,0)$ instead of \eqref{mellinrep-mainline}. Note that unlike \eqref{mellinrep-mainline}, there is no additional $+1$ term here (since $\Gamma(w)\sin(\frac{\pi w}{2})$ has no pole at $w=0$). Thus we would need to study the integral
\[
\frac{1}{2\pi i}  \int_{(-\frac34)} \Gamma(w) \sin\left(\tfrac{\pi w}{2}\right)  \left(\frac{q}{2\pi}\right)^w\zeta^2(\thalfps+w) dw
 \]
instead of the integral in \eqref{h2interchange}, with no additional terms (thus no main terms appear so far).  Moving the line of integration left to $\Re(w)=-c$ crosses poles at the negative odd integers and at $\half-s$. The residue at $w=\half-s$ is
\[
\Gamma(\thalfms) \sin\left(\tfrac{\pi }{4}-\tfrac{\pi s }{2}\right)  \left(\tfrac{q}{2\pi}\right)^{\half-\frac{s}{2}} \Big( \tfrac{ \Gamma'(\halfms)}{ \Gamma(\halfms)} + \tfrac{\pi}{2} \cot \left(\tfrac{\pi }{4}-\tfrac{\pi s }{2}\right) + \log \left(\tfrac{q}{2\pi}\right) + 2\gamma \Big),
\]
instead of \eqref{r1res}. Evaluating at $s=0$ yields the leading main term in \eqref{resulteqn-odd} and explains the $+\frac{\pi}{2}$ term there instead of $-\frac{\pi}{2}$ as in \eqref{resulteqn}. The residues at the negative odd integers contribute lower main terms of size $O(q^{-\half})$ to  \eqref{resulteqn-odd}. The shifted integral positioned at $\Re(w)=-c$ will be discussed at the end of the next section.

\section{The shifted integral}

The goal of this section is to show that
\begin{lemma} \label{final}
There exists a meromorphic function $\mathcal{I}(s,q)$ defined for $s\in \mathcal{S}_2$ such that $\mathcal{I}(s,q)= I(s,q)$ for $s\in \mathcal{S}_1$. At $s=0$, we have
\begin{align}
\label{finallem} \mathcal{I}(0,q)=  \zeta^2(\thalf)+ \sum_{n=1 }^{N} r_n q^{-n}+O_N(q^{-N-1}).
\end{align}
for any integer $N\ge 1$ and some constants $r_n$.
\end{lemma}
We now embark on the proof of Lemma \ref{final}. Let $s\in\mathcal{S}_1$. Applying the functional equation \eqref{fe-2} to $\zeta^2(\half+s+w)$ and the reflection formula to $\Gamma(w)$, we get
\begin{align*}
&I(s,q) = \frac{1}{2\pi i}  \int\limits_{(-c)}   \frac{  (\frac{q}{2\pi})^w (2\pi)^{2(s+w)}  }{ \Gamma(1-w) \cos(\tfrac{\pi}{2} (1-w))}  \cos^2 (\tfrac{\pi}{2}( \thalf-s-w)  )  \Gamma^2 (\thalf-s-w) \zeta^2 (\thalf-s-w) dw\\
&=  \intt \frac{     (\frac{q}{2\pi})^{-c+it} (2\pi)^{2(s-c+it)-1}  }{ \Gamma(1+c-it) \cos(\tfrac{\pi}{2} (1+c-it))}  \cos^2 (\tfrac{\pi}{2}( \thalf-s+c-it)  )  \Gamma^2 (\thalf-s+c-it)   \zeta^2 (\thalf-s+c-it)    dt.
\end{align*}
Since $\thalf-\Re(s)+c  >1$, we can expand
\[
\zeta^2 (\thalf-s+c-it) = \sum_{n\ge 1}  \frac{d(n)}{n^{\half-s+c-it}} 
\]
into an absolutely convergent sum. Exchanging the order of summation and integration, we get
\begin{align*}
I(s,q)= \sum_{n\ge 1} I(s,q,n),
\end{align*}
where
\begin{multline*}
I(s,q,n) =\intt   \frac{     (\frac{q}{2\pi})^{-c+it} (2\pi)^{2(s-c+it)-1}  }{ \Gamma(1+c-it) \cos(\tfrac{\pi}{2} (1+c-it))}  \cos^2 (\tfrac{\pi}{2}( \thalf-s+c-it)  )  \Gamma^2 (\thalf-s+c-it)   \frac{d(n)}{n^{\half-s+c-it}}    dt.
\end{multline*}
The idea now is to understand the oscillatory and stationary behavior of the integrand, and use that information to determine the integral. Such stationary phase analysis can be found in the literature to determine the integral up to a suitable error term (see for example \cite[Proposition 8.2]{bky}). However we cannot simply quote the literature because we need to first establish meromorphic continuation to a neighborhood of $s=0$, {\it after} which we are permitted to determine the value at $s=0$ up to a suitable error term. Therefore, we proceed by revisiting the steps of stationary phase analysis, being careful to obtain meromorphic continuation.

We study the integral $I(s,q,n)$ over finite dyadic intervals, as follows. Note that for each $n\in\mathbb{N}$, the union 
\[
 \bigcup\limits_{k>-\log_2(2\pi qn)} (- \tfrac32(2\pi qn)2^k, -\tfrac58(2\pi qn)2^k ) \ \bigcup \ (-10,10)  \bigcup\limits_{k>-\log_2(2\pi qn)} (\tfrac58(2\pi qn)2^k ,  \tfrac32(2\pi qn)2^k).
\]
forms an open cover of $(-\infty,\infty)$. Thus there exists a partition of unity of $(-\infty,\infty)$ by a collection of smooth functions of the form $W(\frac{t}{T})$, where $T=T(n)$ assumes one of the values below, and $W(x)$ is supported on 
\begin{align}
\label{cover} \tfrac58  &< x <\tfrac32 \  \  \ \ \text{ for } \ \ T=(2\pi qn)2^k, \ k\in\mathbb{Z}, \ k\ge  -\log_2(2\pi qn),  \\
\nonumber -20   &< x < 20 \ \ \ \text{ for } \ \ T=\thalf, \\
\nonumber  -\tfrac32  &< x < -\tfrac58 \  \ \text{ for }  \ \ T=(2\pi qn)2^k, \ k\in\mathbb{Z}, \ k\ge  -\log_2(2\pi qn),
\end{align}
with $W^{(j)}(x)\ll_j 1$ for any $j\ge 0$. Then we have that
\[
I(s,q,n)=\sum_{T} I_{T} (s,q),
\]
where
\begin{multline*}
I_T(s,q)=   \intt W\Big(\frac{t}{T}\Big) \frac{    (\frac{q}{2\pi})^{-c+it} (2\pi)^{2(s-c+it)-1}  }{ \Gamma(1+c-it) \cos(\tfrac{\pi}{2} (1+c-it))} \\
\times \cos^2 (\tfrac{\pi}{2}( \thalf-s+c-it)  ) \Gamma^2 (\thalf-s+c-it)   \frac{d(n)}{n^{\half-s+c-it}}    dt.
\end{multline*}

We first consider the contribution of $T=\frac12$ to $I(s,q)$. The integral $I_\half(s,q)$ is of bounded length, the integrand is holomorphic for $s\in\mathcal{S}_2$, and the sum $\sum_{n\ge 1}   \frac{d(n)}{n^{\half-\Re(s)+c}}$ is absolutely bounded for $s\in \mathcal{S}_2$. Thus $ \sum_{n\ge 1}  I_\half (s,q)$
analytically continues to $\mathcal{S}_2$, with \[ \sum_{n\ge 1}  I_\half(0,q)\ll q^{-c}.\] 
 
Now we restrict to the case $T\neq \half$ (so that $T\ge 1$).  We apply Lemma \ref{stirling} with $M=2c$ for $\Gamma(1+c-it) \cos(\tfrac{\pi}{2} (1+c-it))$ and $\cos (\tfrac{\pi}{2}( \thalf-s+c-it)  ) \Gamma (\thalf-s+c-it)$, which is permissible since $1+c>1$, $ \thalf-\Re(s)+c>1$, and $|t|> \frac58$. These conditions hold for $s$ in the wider set $\mathcal{S}_2$ too. Thus we get
 \begin{align*}
I_T(s,q)= (2\pi q)^{-c}  \frac{d(n)}{n^{\half-s+c}}  \int_{-\infty}^{\infty}  W\Big(\frac{t}{T}\Big)   \Big(\sum_{0\le m\le {2c}} a_m(s) t^{-m} + E(s,t) \Big)     |t|^{c-2s-\half} \exp(i\Phi(t))dt,
\end{align*}
where 
\[
\Phi(t)=-t\log \left( \frac{|t|}{2\pi e qn }\right),
\]
and $a_m(s)$ and $E(s,t)$ are holomorphic functions of $s\in \mathcal{S}_2$ such that 
\begin{align}
\label{a0s} a_m(s)\ll 1, \ \ \ E(s,t) \ll |t|^{-2c-1}, \ \ \ a_0(s) = (2\pi)^{2s-1} (\tfrac{\pi}{2})^\half e^{i \frac{\pi}{4}\mathrm{sgn}(t) }.
\end{align}
We point out here that since we applied Stirling's approximation to $\Gamma(z-it)\cos(\frac{\pi}{2}(z-it))$, we get by \eqref{a0val} that $a_0(s)$ has the factor $e^{-i \frac{\pi}{4}\mathrm{sgn}(-t)}=e^{i \frac{\pi}{4}\mathrm{sgn}(t)}$.

First, we address the contribution of $E(s,t)$ to $I(s,q)$. This equals
 \begin{align}
\label{R-contribution} \sum_{n\ge 1} \sum_{T \ge 1} (2\pi q)^{-c}  \frac{d(n)}{n^{\half-s+c}}  \int_{-\infty}^{\infty}  W\Big(\frac{t}{T}\Big)   E(s,t)      |t|^{c-2s-\half} \exp(i\Phi(t))dt,
\end{align}
which is bounded for $s\in \mathcal{S}_2$ by the absolutely convergent sum
 \begin{align*}
\sum_{n\ge 1} \sum_{T \ge 1} q^{-c}  \frac{d(n)}{n^{\half-\Re(s)+c}} T  \cdot T^{-2c-1}      T^{c-2\Re(s)-\half} \ll q^{-c} \sum_{n\ge 1}  \frac{d(n)}{n^{\half-\Re(s)+c}}   \sum_{T \ge 1}   T^{-c} \ll q^{-c}.\end{align*}
Thus \eqref{R-contribution} has analytic continuation to  $\mathcal{S}_2$ with value at $s=0$ bounded by $O(q^{-c})$. 

It remains to consider
\[
\sum_{n\ge 1} \ \sum_{0\le m\le 2c}  \ \sum_{T \ge 1} a_m(s) I_{T,m}(s,q),
\] 
where we define
 \begin{align}
\nonumber I_{T,m}(s,q)&:= (2\pi q)^{-c}  \frac{d(n)}{n^{\half-s+c}}  \int_{-\infty}^{\infty}  W\Big(\frac{t}{T}\Big)    t^{-m} |t|^{c-2s-\half} \exp(i\Phi(t))dt\\
\label{ITm} &=  (2\pi q)^{-c}   \frac{d(n)}{n^{\half-s+c}} T  \int_{-\infty}^{\infty}  W(x)   (xT)^{-m} |xT|^{c-2s-\half} \exp(i\Phi(xT) )dx.
\end{align}
The last equality follows by the substitution $x=\frac{t}{T}$. We now assume that $W(x)$ is supported on the positive real axis because the negative case is treated similarly. Thus $W(x)$ is supported on $(\frac58, \frac32)$. 

\

{\it Case I.} $T\neq 2\pi q n$. In this case we have $ \frac{T}{2\pi  qn } \ge 2$ or $ \frac{T}{2\pi  qn } \le \half$ by definition \eqref{cover}.  So
\[
\Big| \frac{d}{dx} i \Phi(xT) \Big| =  \Big| \frac{d}{dx} \Big(-i x T \log \left( \frac{xT}{2\pi e qn }\right)\Big)  \Big| = \Big| - i T \log \Big( x \frac{T}{2\pi  qn }\Big) \Big| \gg T
\]
for $x\in (\frac58, \frac32)$ and
\[
\frac{d^j}{dx^j} \Big( - i  \log \Big( x \frac{T}{2\pi  qn }\Big) \Big) =  i (j-1)! (-x)^{-j}  \asymp 1 \ \  \text{ for } j\ge 1.
\]
We can write
\begin{align*}
  \int_{-\infty}^{\infty}  W(x)   x^{-m+c-2s-\half} \exp(i\Phi(xT) )dx =    \int_{-\infty}^{\infty}  W(x)   x^{-m+c-2s-\half} \frac{ - i T \log ( x \frac{T}{2\pi  qn })}{ - i T \log ( x \frac{T}{2\pi  qn })} \exp(i\Phi(xT) )dx
 \end{align*}
 and then repeatedly integrate by parts to get
 \begin{align*}
 \int_{-\infty}^{\infty}  W(x)   x^{-m+c-2s-\half} \exp(i\Phi(xT) )dx 
& = \frac{1}{T}  \int_{-\infty}^{\infty}  \frac{d}{dx}\Big( \frac{W(x)   x^{-m+c-2s-\half} } { i  \log ( x \frac{T}{2\pi  qn }) } \Big) \exp(i\Phi(xT) )dx\\
& = \frac{1}{T^2}  \int_{-\infty}^{\infty} \frac{d}{dx} \Bigg( \frac{ \frac{d}{dx}\Big( \frac{W(x)   x^{-m+c-2s-\half} } {  i  \log ( x \frac{T}{2\pi  qn }) } \Big) }{  i  \log ( x \frac{T}{2\pi  qn }) } \Bigg) \exp(i\Phi(xT) )dx,
 \end{align*}
 and so on, with each of these successive integrals an entire function of $s$. In this way we see that for $s\in\mathcal{S}_2$, we have 
 \[
I_{T,m}(s,q)\ll_{k} q^{-c} T^{-k}
 \]
  for any $k\ge 0$. This ensures the absolute convergence of the sum
  \begin{align*}
\sum_{n\ge 1} \ \sum_{0\le m\le 2c}  \sum_{\substack{T \ge 1 \\ T\neq 2\pi q n}} a_m(s) I_{T,m}(s,q)
\end{align*}
for $s\in \mathcal{S}_2$ and therefore its analytic continuation to $\mathcal{S}_2$ with value at $s=0$ bounded by $O(q^{-c})$. 

\

{\it Case II.} $T= 2\pi  n q$. By construction \eqref{cover}, we are considering the unique function $W(\frac{t}{2\pi n q})$ in this partition which contains a neighborhood of $2\pi n q$ in its support. We have, using definition \eqref{ITm},
\begin{align*}
I_{2\pi  n q,m}(s,q)&= (2\pi q)^{-c}  \frac{d(n)}{n^{\half-s+c}} 2\pi n q \int_{-\infty}^{\infty}  W(x)   (x2\pi n q)^{-m+c-2s-\half} \exp(i\Phi(x2\pi n q ) )dx\\
&=  (2\pi  q)^{\half-m-2s}    \frac{d(n)}{n^{s+m}}  \int_{-\infty}^{\infty}  W(x)  x^{-m+c-2s-\half}  \exp(i\phi(x ) )dx,
\end{align*}
where
\[
\phi(x)= -2\pi n q x \log\Big(\frac{x}{e}\Big).
\]

We have
\[
\phi'(x)= -2\pi n q \log x, \ \  \ \phi^{(j)}(x)= 2\pi n q (j-2)! (-x)^{-j+1} \ \text{ for } j\ge 2. 
\]
Note that $\phi'(1)=0$ (so that $x=1$ is a stationary point). Taking a Taylor series expansion about $x=1$, we get
\[
\phi(x) = 2\pi n q\Big(1 - \frac{(x-1)^2}{2} + \frac{(x-1)^3}{6}-\ldots\Big)
\]
for $x\in (\frac58, \frac32)$, so that
\begin{align}
\label{phi-simp} \exp(i\phi(x ) ) = \exp\Big(  2\pi i n q\Big( - \frac{(x-1)^2}{2} + \frac{(x-1)^3}{6}- \ldots\Big)\Big).
\end{align}
This uses the nice simplification $e^{2\pi i nq}=1$.

We will evaluate $I_{2\pi  n q,m}(s,q)$ by stationary phase analysis.  As is the usual technique, we first isolate $x=1$ further. Let $w(x)$ be a smooth function such that
\begin{align*}
&w(x) = 1 \ \ \text{ for } 1-(nq)^{-\frac{5}{12}}< x <  1+(nq)^{-\frac{5}{12}},\\
&w(x) = 0 \ \ \text{ for } x< 1-2(nq)^{-\frac{5}{12}} \ \text{ or } x> 1+2(nq)^{-\frac{5}{12}},\\
&w^{(j)}(1) = 0 \ \ \text{ for } j\ge 1, \ \ \ w^{(j)}(x) \ll_j ((nq)^{\frac{5}{12}})^j  \ \ \text{ for } j\ge 0.
\end{align*}
We split up the integral as $I_{2\pi  n q,m}(s,q)= I_{2\pi  n q,m}^w(s,q)+ I_{2\pi  n q,m}^{1-w}(s,q)$, where
\begin{align*}
& I_{2\pi  n q,m}^w(s,q):=  (2\pi q) ^{\half-m-2s}   \frac{d(n)}{n^{s+m}}  \int_{-\infty}^{\infty}  W(x) w(x) x^{-m+c-2s-\half}  \exp(i\phi(x ) )dx,\\
& I_{2\pi  n q,m}^{1-w}(s,q):=  (2\pi  q)^{\half-m-2s}   \frac{d(n)}{n^{s+m}}  \int_{-\infty}^{\infty}  W(x) (1-w(x))  x^{-m+c-2s-\half}  \exp(i\phi(x ) )dx.
\end{align*}
Thus the integrand in $I_{2\pi  n q,m}^w(s,q)$ is supported very close to $x=1$ while the integrand in $I_{2\pi  n q,m}^{1-w}(s,q)$ is supported some distance away from $x=1$. 

Repeatedly integrating by parts in $I_{2\pi  n q,m}^{1-w}(s,q)$, and using \eqref{phi-simp}, we get
\begin{align*}
&\int_{-\infty}^{\infty}  W(x) (1-w(x))  (2\pi x )^{-m+c-2s-\half}   \exp(i\phi(x ) ) dx \\
&= \frac{1}{2\pi i n q}\int_{-\infty}^{\infty} \frac{d}{dx} \Big( \frac{W(x) (1-w(x))  (2\pi x )^{-m+c-2s-\half}}{  - (x-1) + \frac{(x-1)^2}{2}-\ldots }  \Big) \exp(i\phi(x ) ) dx \\
&= \frac{1}{(2\pi i n q)^2}\int_{-\infty}^{\infty} \frac{d}{dx} \Bigg( \frac{\frac{d}{dx} \Big( \frac{W(x) (1-w(x))  (2\pi x )^{-m+c-2s-\half}}{  - (x-1) + \frac{(x-1)^2}{2}-\ldots }}{{  - (x-1) + \frac{(x-1)^2}{2}-\ldots }}  \Bigg) \exp(i\phi(x ) ) dx,
\end{align*}
and so on, with each of these successive integrals an entire function of $s$. Since $ (nq)^{-\frac{5}{12}}\ll |x-1|\le \frac12$ in the support of $W(x)(1-w(x))$, and $w^{(j)}(x) \ll_j (nq)^{\frac{5}{12}j}$ for $j\ge 0$, we see that for $s\in\mathcal{S}_2$, after integrating by parts $k\ge 0$ times we have 
\[
\sum_{\substack{n\ge 1}} \ \sum_{0\le m\le 2c}   a_m(s) I_{2\pi n q,m}^{1-w}(s,q) \ll_k   \sum_{n\ge 1}  \ \sum_{0\le m\le 2c}   q^{\half-m-2\Re(s)}   \frac{d(n)}{n^{\Re(s)+m}} \Big(\frac{(nq)^{\frac{5}{6}}}{nq }\Big)^k \ll q^{-\frac{7c}{6}+1},
\]
where the last bound follows by taking $k=7c$. This ensures the absolute convergence of the sum above and therefore its analytic continuation to $\mathcal{S}_2$ with value at $s=0$ bounded by $O(q^{-c})$. 

It remains to consider $I_{2\pi  n q,m}^{w}(s)$, which we write as, using \eqref{phi-simp},
\begin{align}
\label{ITm2} I_{2\pi  n q,m}^w(s,q) = (2\pi  q)^{\half-m-2s}   \frac{d(n)}{n^{s+m}}  \int_{-\infty}^{\infty}  h(s,x)
 \exp( -i \pi n q (x-1)^2 ) dx,
\end{align}
where
\begin{align}
\label{hsx} h(s,x) =  W(x)  w(x) x^{-m+c-2s-\half}  e^{i2\pi n q(  \frac{(x-1)^3}{3}-\ldots)} .
\end{align}
Let 
\[
\hat{h}(s,y)= \intt h(s,x) e^{-2\pi i yx} dx
\]
denote the Fourier transform of $h$. Integrating by parts $k\ge 0$ times, we have for $s\in\mathcal{S}_2$ and $|y|>1$,
\begin{align}
\label{hk} \hat{h}(s,y) = \hat{h}_k(s,y) := \intt \Big( \frac{\partial^k}{\partial x^k} h(s,x)  \Big) \frac{e^{-2\pi i yx} }{(2\pi i y)^k }   dx 
\ll_k  \Big (\frac{(nq)^{\frac{5}{12}}}{|y|}\Big)^{k}.
\end{align}
This uses $w^{(j)}(x) \ll_j (nq)^{\frac{5}{12}j}$ for $j\ge 0$ and $n q |x-1|^2\ll (nq)^{\frac{5}{12}}$ in the support of $w(x)$.
For $|y|\le 1$, we have the trivial bound $\hat{h}(s,y) \ll 1$. Also note that $\hat{h}(s,y)$ (for $|y|\le 1$) and $\hat{h}_k(s,y) $ (for $|y|> 1$) are entire functions of $s$. By Fourier inversion, we have
\begin{align*}
& \int_{-\infty}^{\infty}  h(s, x)
e^{ -i \pi n q (x-1)^2 } dx =
 \lim_{\alpha\to\infty}  \int_{-\alpha}^{\alpha}  h(s,x) e^{ -i \pi n q (x-1)^2 } dx  \\
  =&\lim_{\alpha\to\infty} \int_{-\alpha}^{\alpha}  \Big( \intt \hat{h}(s,y) e^{2\pi i x y}  dy\Big)  e^{-i \pi n q (x-1)^2 } dx
=\lim_{\alpha\to\infty} \int_{-\alpha}^{\alpha}  \Big( \int_{-\alpha}^{\alpha} \hat{h}(s,y) e^{2\pi i x y}dy  \Big)  e^{-i \pi n q (x-1)^2 } dx,
  \end{align*}
  where the last line uses the upper bound for $\hat{h}(s,y)$ given in \eqref{hk}. We can exchange the order of integration to get
 \begin{align*}
 \lim_{\alpha\to\infty}   \int_{-\alpha}^{\alpha}  \hat{h}(s,y)  \int_{-\alpha}^{\alpha} e^{2\pi i  y x}   e^{-i \pi n q (x-1)^2  }dx  \ dy= \lim_{\alpha\to\infty}  \int_{-\alpha}^{\alpha}  \hat{h}(s,y) \int_{-\infty}^{\infty} e^{2\pi i  y x}   e^{ -i \pi n q (x-1)^2 } dx \ dy,
\end{align*}
where the inner integral was extended to infinity because for $|y|\le \alpha$,
\begin{align*}
 \int_{\alpha}^{\infty} &e^{2\pi i  y x -i \pi n q (x-1)^2 } dx  = \lim_{\beta\to\infty} e^{2\pi i  y  + i \pi \frac{y^2}{nq} }   \int_{\alpha}^{\beta}  e^{-i\pi nq (x-1- \frac{y}{nq})^2 } dx
 \end{align*}
 by completing the square, and  this equals
 \begin{align*}
\lim_{\beta\to\infty} e^{2\pi i  y  + i \pi \frac{y^2}{nq} } \Bigg(  \frac{e^{-i\pi nq (x-1- \frac{y}{nq})^2}}{-2i\pi nq (x-1- \frac{y}{nq})}  \Bigg|_{x=\alpha}^{x=\beta}  -  \int_{\alpha}^{\beta} \frac{e^{-i\pi nq (x-1- \frac{y}{nq})^2}}{2i\pi nq (x-1- \frac{y}{nq})^2} dx\Bigg) \ll \alpha^{-1}
 \end{align*}
by integrating by parts. Now,
 \begin{align*}
&\lim_{\alpha\to\infty}   \int_{-\alpha}^{\alpha}  \hat{h}(s,y) \int_{-\infty}^{\infty} e^{2\pi i  y x}   e^{ -i \pi n q (x-1)^2 } dx \ dy \\ =& \lim_{\alpha\to\infty}  \int_{-\alpha}^{\alpha}  \hat{h}(s,y) \frac{ e^{i \pi ( 2  y+  \frac{y^2}{nq}-\frac{1}{4}) }}{\sqrt{nq}} \ dy 
=\frac{ e^{-\frac{i\pi}{4}}}{\sqrt{nq}}  \int_{-\infty}^{\infty}  \hat{h}(s,y) e^{ 2 \pi    y i } e^{i \pi  \frac{y^2}{nq} }\ dy.
 \end{align*}
 by evaluating the Gaussian $x$-integral using \cite[equations 17.23.16, 17.23.17]{gr} and then extending the $y$-integral to infinity using \eqref{hk}.
 Writing
 \[
E(y) =\sum_{j=0}^{6c} \frac{1}{j!}\Big(  \frac{i\pi y^2}{nq} \Big)^j,
 \]
we have
 \begin{multline}
\label{taylor} \frac{ e^{-\frac{i\pi}{4}}}{\sqrt{nq}}  \int_{-\infty}^{\infty}  \hat{h}(s,y) e^{ 2 \pi    y i } e^{i \pi  \frac{y^2}{nq} }\ dy 
= \frac{ e^{-\frac{i\pi}{4}}}{\sqrt{nq}}  \int_{-\infty}^{\infty}  \hat{h}(s,y) e^{ 2 \pi    y i }  E(y)  dy 
\\
+ \frac{ e^{-\frac{i\pi}{4}}}{\sqrt{nq}}  \int_{-1}^{1}  \hat{h}(s,y) e^{ 2 \pi    y i } \Big(e^{i \pi  \frac{y^2}{nq} } -  E(y)\Big)  dy
 + \frac{ e^{-\frac{i\pi}{4}}}{\sqrt{nq}}  \int_{|y|>1}  \hat{h}_{k}(s,y) e^{ 2 \pi    y i } \Big(e^{i \pi  \frac{y^2}{nq} } -  E(y)\Big) \ dy,
 \end{multline}
where we used \eqref{hk} for the last integral. All the integrals on the right hand side of \eqref{taylor} are absolutely convergent, so they are entire functions of $s$. For the second and third integrals, by Taylor's theorem applied to the function $e^{ix}$, we have the following bounds for $s\in\mathcal{S}_2$:
\begin{align}
\label{bounds} &\frac{ e^{-\frac{i\pi}{4}}}{\sqrt{nq}}  \int_{-1}^{1}  \hat{h}(s,y) e^{ 2 \pi    y i } \Big(e^{i \pi  \frac{y^2}{nq} } -  E(y)\Big)  dy \ll \frac{1}{\sqrt{nq}} \Big(\frac{1}{nq}\Big)^{6c+1}  \ll \Big(\frac{1}{nq}\Big)^{6c+\frac32} , \\
\nonumber &\frac{ e^{-\frac{i\pi}{4}}}{\sqrt{nq}}  \int_{|y|>1}  \hat{h}_{k}(s,y) e^{ 2 \pi    y i } \Big(e^{i \pi  \frac{y^2}{nq} } -  E(y)\Big) \ dy \ll \frac{1}{\sqrt{nq}} \int_{|y|>1} \Big (\frac{(nq)^{\frac{5}{12}}}{|y|}\Big)^{12c+4} \Big(\frac{|y|^{2}}{nq}\Big)^{6c+1} dy \ll \Big(\frac{1}{nq}\Big)^{c+\frac16}, 
\end{align}
where we used \eqref{hk} with $k=12c+4$. Plugging back \eqref{taylor} into \eqref{ITm2}, we get that 
\begin{align}
 \label{plugback} &\sum_{\substack{n\ge 1\\0\le m\le 2c }}   a_m(s) I_{2\pi  n q,m}^w(s) =\sum_{\substack{n\ge 1\\ 0\le m\le 2c}}  a_m(s)  \sqrt{2\pi}  e^{-\frac{i\pi}{4}} (2\pi q)^{-m-2s}  \frac{d(n)}{n^{\half+s+m}}   \int_{-\infty}^{\infty}  \hat{h}(s,y) e^{ 2 \pi    y i }  E(y)  dy \\
\nonumber &+ \sum_{\substack{n\ge 1\\0\le m\le 2c }}    a_m(s) \sqrt{2\pi} e^{-\frac{i\pi}{4}} (2\pi q)^{-m-2s}   \frac{d(n)}{n^{\half+s+m}}   \int_{-\infty}^{\infty}  \hat{h}(s,y) e^{ 2 \pi    y i } \Big(e^{i \pi  \frac{y^2}{nq} } -  E(y)\Big)  dy. 
 \end{align}
In light of the bounds \eqref{bounds}, the second sum on the right hand side of \eqref{plugback} analytically continues to $\mathcal{S}_2$, with value at $s=0$ being bounded by $O(q^{-c+1})$. 

It remains to obtain the meromorphic continuation from $\mathcal{S}_1$ to $\mathcal{S}_2$ of 
\begin{align}
\label{lastone} \sum_{\substack{n\ge 1}} \ \sum_{0\le m\le 2c} \ \sum_{0\le j \le 6c}  a_m(s) I_{2\pi n q,m,j}^w (s,q), 
\end{align}
where
\begin{align*}
I_{2\pi n q,m,j}^w (s,q) :=& \sqrt{2\pi}  e^{-\frac{i\pi}{4}} (2\pi q)^{-m-2s}\frac{d(n)}{n^{\half+s+m}}   \int_{-\infty}^{\infty}  \hat{h}(s,y) e^{ 2 \pi   i y }  \frac{1}{j!}\Big(  \frac{i\pi y^2}{nq} \Big)^j dy.
\end{align*}
By Fourier inversion, we have
\[
 \int_{-\infty}^{\infty}  \hat{h}(s,y) e^{ 2 \pi   i y  }  y^{2j} dy = \frac{1}{(2\pi i )^{2j}} \frac{\partial^{2j}}{\partial x^{2j}} h(s,x)\bigg|_{x=1}.
\] 
Thus keeping in mind definition \eqref{hsx}, we have
\[
I_{2\pi n q,m,j}^w (s,q) = \sqrt{2\pi}  e^{-\frac{i\pi}{4}}  (2\pi q)^{-m-2s}   \frac{d(n)}{n^{\half+s+m}} P_{m,j}\Big(s,\frac{1}{nq}\Big),
\]
where $P_{m,j}(s,\frac{1}{nq})$ is entire in $s$ and polynomial in $\frac{1}{nq}$, and the degree of this polynomial is strictly greater than $0$ for $j>0$. For $j=0$, we note that $P_{m,0}(s,\frac{1}{nq})=h(s,1)=1$. Summing over $n$, we get that \eqref{lastone} equals
\[
\sum_{0\le m\le 2c} \ \sum_{0\le j \le 6c}  a_m(s) \sqrt{2\pi}  e^{-\frac{i\pi}{4}}  (2\pi q)^{-m-2s}   P_{m,j}(s,\tfrac{1}{nq}) \zeta^2(\thalf+s+m).
\]
This expression continues to a meromorphic function on $\mathcal{S}_2$ (because the Riemann Zeta function does) and we have that its value at $s=0$ equals, using \eqref{a0s}, 
\begin{align}
\label{Wpos} a_0(0) \sqrt{2\pi}  e^{-\frac{i\pi}{4}}   \zeta^2(\thalf) +  \sum_{n=1 }^{N} r_n' q^{-n}+O_N(q^{-N-1})= \frac{1}{2}  \zeta^2(\thalf) +  \sum_{n=1 }^{N} r_n' q^{-n}+O_N(q^{-N-1}), 
\end{align}
for some constants $r_n'$ and any $N\ge 1$, by taking $c$ large enough. 

Recall that for this calculation, we have been assuming that $W(x)$ is supported on the positive real axis. In the case that the support of $W(x)$ is on the negative real axis, there is a stationary point at $x=-1$. In this case the leading term from stationary phase is $a_0(0) \sqrt{2\pi}  e^{\frac{i\pi}{4}}   \zeta^2(\thalf)$, where $a_0(0)=(2\pi)^{-1}(\frac{\pi}{2})^{\half} e^{-\frac{i\pi}{4}} $ by \eqref{a0s}, so that  we obtain 
\begin{align}
\label{Wneg} \frac{1}{2}  \zeta^2(\thalf) +  \sum_{n=1 }^{N} r_n'' q^{-n}+O_N(q^{-N-1})
 \end{align}
 for some constants $r_n''$.
 Adding together \eqref{Wpos} and \eqref{Wneg}, we get \eqref{finallem}.
\endproof

{\it The sum over the odd characters.} The treatment of the integral
\begin{align}
\label{sin-int} \frac{1}{2\pi i}  \int_{(-c)} \Gamma(w) \sin\left(\tfrac{\pi w}{2}\right)  \left(\frac{q}{2\pi}\right)^w\zeta^2(\thalfps+w) dw,
 \end{align}
 required to obtain \eqref{resulteqn-odd}, follows similar calculations, but we highlight some important differences with the treatment of the integral $I(s,q)$. Instead of \eqref{cos}, we have
 \begin{align*}
 \sin(\tfrac{\pi}{2}(z+it)) &= \tfrac{1}{2i} \big( e^{i(\frac{\pi}{2}(z+it))} - e^{-i(\frac{\pi}{2}(z+it))}\big)= i  \mathrm{sgn}(t) \tfrac{1 }{2} e^{-i  z\tfrac{\pi}{2} \mathrm{sgn}(t) + \tfrac{\pi}{2} |t|} \big(1 - e^{iz \pi   \mathrm{sgn}(t)  -\pi |t|} \big).
\end{align*}
Thus the leading terms in Stirling's estimates for $ \Gamma(w) \sin\left(\tfrac{\pi w}{2}\right)$ and $ \Gamma(w) \cos\left(\tfrac{\pi w}{2}\right)$ differ only by a factor of $i  \mathrm{sgn}(t) $, and this extra factor is carried through the entire calculation. In the end we get that the value at $s=0$ of the leading term in the meromorphic continuation of the integral \eqref{sin-int} equals
\begin{align}
\label{end-sin}i \frac{ \zeta^2(\thalf)}{2}   -i \frac{ \zeta^2(\thalf)}{2}   = 0,
\end{align}
from the two stationary points (one positive and one negative).

\

{\bf Acknowledgments.} We are grateful to the anonymous referee for many insightful suggestions that helped to improve the manuscript -- in particular for comments regarding reciprocity formulae and for pointing out a mistake in the previous version of the manuscript. We also thank Agniva Dasgupta for a careful reading and valuable remarks.

\bibliographystyle{amsplain} 

\bibliography{bibliography}

\end{document}